\documentclass[10pt,a4paper]{amsart}
\usepackage{latexsym}
\usepackage{amssymb}
\usepackage{enumerate}
\usepackage[all]{xy}
\usepackage{ifpdf}

\ifpdf
        \usepackage[pdftex]{graphicx}
        \DeclareGraphicsExtensions{.pdf}
        \usepackage{hyperref}
\else
        \usepackage{graphicx}
        \DeclareGraphicsExtensions{.eps}
        \usepackage{hyperref}
\fi

\newtheorem{theorem}{Theorem}[section]

\newtheorem{proposition}[theorem]{Proposition}
\newtheorem{lemma}[theorem]{Lemma}
\newtheorem{corollary}[theorem]{Corollary}

\theoremstyle{remark}
\newtheorem{remark}[theorem]{Remark}

\theoremstyle{definition}
\newtheorem{definition}[theorem]{Definition}

\newtheorem{example}[theorem]{Example}

\newcommand\leftmap[1]{\smash{\mathop{\leftarrow}\limits^{#1}}}

\newcommand\cC{{\mathcal C}}

\newcommand\cP{{\mathcal P}}

\newcommand\cF{{\mathcal F}}

\newcommand\KK{{\mathbb K}}
\newcommand\CC{{\mathbb C}\,}

\newcommand\ZZ{{\mathbb Z}}

\newcommand\NN{{\mathbb N}}

\newcommand\PP{{\mathbb P}}

\newcommand\gl{{\lambda}}

\newcommand\gk{{\kappa}}

\DeclareMathOperator{\Sing}{\operatorname{Sing}}

\newcommand{\C}{\mathcal C}

\newcommand{\SP}{\text{\rm SP}}

\DeclareMathOperator{\corank}{\operatorname{corank\ }}

\DeclareMathOperator{\codim}{\operatorname{codim\ }}
\DeclareMathOperator{\rank}{\operatorname{rank\ }}

\title[The Max Noether Fundamental Theorem is Combinatorial]{The Max Noether Fundamental Theorem is Combinatorial}

\author[J.I.Cogolludo Agust{\'\i}n]{J.I. Cogolludo Agust{\'\i}n}
\address{Departamento de Matem\'aticas, IUMA\\
Universidad de Zaragoza\\
C/ Pedro Cerbuna, 12\\
E-50009 Zaragoza, Spain
}
\email{jicogo@unizar.es}

\author[M.\'A.Marco Buzun\'ariz]{M.\'A.Marco Buzun\'ariz}
\address{ICMAT: CSIC-Complutense-Autonoma-Carlos III\\
Departamento de \'Algebra - Facultad de CC. Matem\'aticas -
Plaza de las Ciencias, 3\\
28040 Madrid, Spain
}
\email{mmarco@unizar.es}

\keywords{Plane algebraic curves, singularities}
\thanks{}
\subjclass[2000]{58K65, 14Q05, 14B05, 14C22, 14H50, 14F25, 32K07}

\begin{document}

\begin{abstract}
In the present paper we give a reformulation of the Noether Fundamental Theorem 
for the special case where the three curves involved have the same degree.
In this reformulation, the local Noether's Conditions are weakened.
To do so we introduce the concept of Abstract Curve Combinatorics (ACC) which 
will be, in the context of plane curves, the analogue of matroids for hyperplane arrangements.
\end{abstract}

\maketitle

\section{Introduction}
In 1873 Noether stated his celebrated Fundamental Theorem~\cite{Noether-fundamental-theorem}, 
sometimes referred to as the ``$AF+BG$'' Theorem. This theorem brings together the geometric and algebraic 
conditions plane projective algebraic curves should satisfy when belonging to a pencil. The following 
statement can be found in~\cite{Fulton-algebraic-curves}.

\begin{theorem}[Max Noether's Fundamental Theorem]
Let $F,G,H$ be homogeneous reduced polynomials in three variables defining projective algebraic curves
$V(F)$, $V(G)$, and $V(H)$. Assume $V(F)$ and $V(G)$ have no common components. 
Then there is an equation $H=AF+BG$ (with $A,B$ forms of degrees $\deg(H)-\deg(F)$ and $\deg(H)-\deg(G)$ respectively)
if and only if $H_P\in (F_P,G_P)\subset \mathcal O_P(\PP^2)$ for any $P\in V(F)\cap V(G)$.
\end{theorem}

Here we denote by $F_P$ the germ of $F$ at $P$, by $(F_P,G_P)$ the local ideal generated by the germs
$F_P$ and $G_P$, and by $V(F)\subset \PP^2$ the set of zeroes of $F$. The local conditions on the 
equations $F,G,H$ are called \emph{Noether conditions}.

This theorem was originally attacked both from geometric and algebraic points of view 
(\cite{Scott-proof,Fulton-algebraic-curves}) and it has been
recently generalized to the non-reduced case by Fulton~\cite{Fulton-adjoints}.

Most of the efforts to understand and rewrite Noether's Fundamental Theorem have been focused on finding 
conditions that are equivalent to the Noether conditions in particular instances like transversality of 
branches, ordinary singularities, etc.

Our purpose here is to concentrate on the case where $\deg F=\deg G=\deg H$ and to weaken the Noether's 
conditions so as to have strictly weaker local conditions that can still provide the equivalence of the 
result. Note that the Noether Fundamental Theorem is a combination of a global condition (the existence of
the curves $F,G,H$) and local conditions. Our weakened local conditions combined with the global condition
result in this equivalence.

The weakened local conditions can be briefly described as follows:
We say $F$ satisfies the \emph{combinatorial conditions} with respect to $G$ and $H$ if for any point 
$P\in V(F)\cap V(GH)$ and any local branch $\delta$ of $F$ at $P$ then $\mu_P(\delta,G)=\mu_P(\delta,H)$, 
where $\mu_P$ denotes the multiplicity of intersection of branches at $P$.
Also we say that $F,G,H$ satisfy the conditions for a \emph{combinatorial pencil} if each equation
satisfies the combinatorial conditions with respect to the other two equations. We also introduce the
concept of a \emph{primitive combinatorial pencil} which corresponds with the geometric idea that the fibers 
of the map over $\PP^1$ induced by the pencil after resolution of indeterminacy are connected. 
In \S\ref{sec-stein} we prove that any combinatorial pencil can be refined to a primitive combinatorial pencil.

The global condition can be rewritten as follows:
If $\deg F=\deg G=\deg H$, then the condition $H=AF+BG$ simply means that $H$ belongs to the pencil
generated by $F$ and $G$, or simply that $F,G,H$ belong to a pencil. 

The main result is the following.

\begin{theorem}
Let $F,G,H$ be projective plane curves of the same degree. Assume $F$ and $G$ have no common components. 
If $F,G,H$ belong to a primitive combinatorial pencil then they belong to a pencil.
\end{theorem}

To end this introduction we present two examples aimed to clarify the sharpness of these combinatorial 
conditions. The first one points out that the combinatorial conditions are indeed weaker than the Noether 
conditions and the second one suggests that the conditions cannot be weakened.

\begin{example}
This first example shows that the (local) Noether Conditions are stronger that the combinatorial condition
described above. Consider the germs $f=x^3$, $g=y^2$ and $h=y^2+(x+y)^3$ in $\mathcal O_{P}(\PP^2)=\CC\{x,y\}$, 
$P=[0:0:1]$. It is obvious that they satisfy the combinatorial conditions at $P$ since 
$\mu_P(f,g)=\mu_P(f,h)=\mu_P(g,h)=6$. However, $h\notin (f,g)$ since $h=y^2+x^3+y^3+3xy^2+3x^2y$, where 
 $h_1=y^2+x^3+y^3+3xy^2\in (x^3,y^2)$, but $h_2=x^2y\notin (x^3,y^2)$, and $h=h_1+h_2$. 
%Of course, the curves
% $F=X^3$, $G=ZY^2$, and $H=ZY^2+(X+Y)^3$ do not satisfy the local conditions at the other intersections
% $[0:1:0]$ and $[0:1:-1]$.
\end{example}

\begin{example}
This second example shows that the combinatorial conditions have to be stated for each branch, as 
opposed to each irreducible component. Consider $F=ZY^2-ZX^2+X^3$, $G=(X+Y)^3$, and $H=(X-Y)^3$, 
three cubics. Note that $V(F)\cap V(G)=V(F)\cap V(H)=V(G)\cap V(H)=\{P=[0:0:1]\}$ and
$\mu_P(F,G)=\mu_P(F,H)=\mu_P(G,H)=9$. However, $F,G,H$ are not in a pencil. Note that the combinaotial 
conditions are \emph{not} satisfied, since $F$ is not locally irreducible at $P$ and the two branches 
$\delta_1$ and $\delta_2$ satisfy $\mu_P(\delta_1,G)=\mu_P(\delta_2,H)=6$, and $\mu_P(\delta_1,H)=\mu_P(\delta_2,G)=3$.
\end{example}

\section{Settings}
\subsection{Abstract Curve Combinatorics}
\begin{definition}
\label{def-acc}
An \emph{Abstract Curve Combinatorics} (ACC for short) is a sextuplet 
$W:=(\mathbf r,S,\Delta,\partial,\phi,\mu),$
where
\begin{enumerate}
\item \label{def-acc-cond1}
$\mathbf r$, $S$, and $\Delta$ are finite sets,
\item \label{def-acc-cond2}
$\partial: \Delta\to S$ and $\phi:\Delta\to \mathbf r$ are surjective maps,
\item \label{def-acc-cond3}
$\mu:\SP^2(\Delta)\to \NN$, where $\SP^2(\Delta)$ is the symmetric product of $\Delta$, such
that $\mu(\delta_1,\delta_2)> 0$ if and only if $\partial(\delta_1)= \partial(\delta_2)$ 
and $\phi(\delta_1)\neq \phi(\delta_2)$.
\end{enumerate}
\end{definition}

For simplicity, we denote $\Delta_P:=\partial^{-1}(P)$, $P\in S$.

We say that two ACC's are equivalent if there are bijections preserving the corresponding maps.

\begin{remark}
Note that any projective curve $\cC\subset \PP^2$ determines naturally an ACC 
$W_\cC:=(\mathbf r,S,\Delta,\partial,\phi,\mu),$ (which will be referred to as the 
\emph{Weak Combinatorial Type of $\cC$}) as follows:
\begin{enumerate}[$(i)$]
\item 
The set $\mathbf r$ is the set of irreducible components of $\mathcal C$, 
\item
The set $S:=\Sing(\mathcal C)$, is the set of singular points of $\mathcal C$, 
\item
$\Delta:=\cup_{P\in S} \{\Delta_P\}$ where $\Delta_P$ is the set of local branches of 
$\mathcal C$ at $P\in S$, $\partial(\delta):=P$ if $\delta\in \Delta_P$, 
and $\phi$ assigns to each local branch the global irreducible component that contains it.
\item
$\mu(\delta_1,\delta_2)$ is defined as the multiplicity of intersection between $\delta_1$ and $\delta_2$
(when $\partial(\delta_1)= \partial(\delta_2)$ and $\phi(\delta_1)\neq \phi(\delta_2)$) and as zero otherwise.
\end{enumerate}
\end{remark}

In accordance with this motivation, given an ACC $W=(\mathbf r,S,\Delta,\partial,\phi,\mu)$, we will refer 
to the elements of $\mathbf r$ (resp. $S$, and $\Delta$) as \emph{irreducible components}, (points, and branches).
Also $\mu$ will be referred to as the \emph{intersection multiplicity} of two branches.

\subsection{B\'ezout Condition and degrees}
Consider $W$ an ACC and define 
$d_{i,j}:=\sum_{{\tiny\begin{matrix} \phi(\delta_1)=i,\\ 
\phi(\delta_2)=j\end{matrix}}} \mu(\delta_1,\delta_2)$, for any $i,j\in \mathbf r, i\neq j$.

\begin{definition}
\label{eq-wc}
$W$ satisfies the \emph{B\'ezout Condition} if 
$\frac{d_{i,j}d_{i,k}}{d_{j,k}}$ is independent of $j,k\in \mathbf r$. 
In that case, one can define
$$
d_i:=+\sqrt{\frac{d_{i,j}d_{i,k}}{d_{j,k}}}.
$$
and will be referred to as the \emph{degree} of $i$.
\end{definition}

Note that the Weak Combinatorial Type of a plane projective curve satisfies the B\'ezout Condition 
and $d_i$ coincides with the algebraic degree of the irreducible component $i$.

\subsection{Combinatorial Pencils}
Let $W$ be an ACC satisfying the B\'ezout Condition.
\begin{definition}
\label{def-comb-pencil}
We say that $W$ contains a combinatorial pencil if there exist
$\bar m:=(m_i)_{i\in \mathbf r}$ a list of integers and
$\cF=\{F_1,\dots,F_k\}$, $k\geq 3$ a partition of $\mathbf r$ such that:
\begin{enumerate}
\item 
$\sum_{i\in F_j} m_id_i$ is independent of $j\in\{1,\dots,k\}$, 
(such constant will be denoted by $d_\cF$) and
\item 
for any $P\in S$ one of the following two conditions is satisfied:
\begin{enumerate} 
\item
\label{def-cond1}
either $\phi(\Delta_P)\subset F_i$ for a certain $i=1,\dots,k$,
\item 
\label{def-cond2}
or $\phi(\Delta_P)\not \subset F_i$, in which case for each $\delta\in \Delta_P$, the 
natural number
$$\sum_{{\tiny \begin{matrix} \phi(\delta')\in F_j \end{matrix}}} 
m_{\phi(\delta')} \mu(\delta,\delta')$$
is independent of $j$ (as long as $\phi(\delta)\notin F_j$).
Such a constant will be denoted by $k_{\delta}$.
\end{enumerate}
\end{enumerate}
The points $P\in S$ satisfying~(\ref{def-cond2}) will be called the 
\emph{base points of the combinatorial pencil} and each $F_i\in \cF$ will be called a \emph{fiber}.
The integer $m_i$ will be called the \emph{multiplicity} of the $i$-th component and the members 
of the partition $\cF$ are the \emph{members} of the pencil.

We also say that three curves $F$, $G$ and $H$ \emph{belong to a combinatorial pencil} if 
$(\{F,G,H\},\bar m)$ is a combinatorial pencil, where $\bar m$ is the list of multiplicities of 
the components of $\C=F\cup G \cup H$.
\end{definition}

Our purpose will be to investigate under what circumstances three curves belonging to a combinatorial pencil, 
also belong to a pencil, that is $H=A F+B G$ for some $A,B\in \CC^*$. 
Note that this is not true in general as one can simply see with line arrangements.
Consider $F=XY$, $G=X^2-Y^2$, and $H=X^2-4Y^2$. It is obvious that $F$, $G$, and $H$ belong to a combinatorial 
pencil, but not to a pencil. The geometrical reason behind this phenomenon is that the resolution of the rational 
map $[X:Y:Z]\mapsto [F:G]$ does not have connected fibers.

One needs an extra condition that assures that the pencil is primitive. The definition of a primitive combinatorial
pencil and the fact that any combinatorial pencil can be reduced to a primitive one will be the main idea 
of the coming section.

\section{Combinatorial version of the Stein Factorization Theorem}
\label{sec-stein}

In our context of pencils in $\PP^2$, the Stein Factorization Theorem (\cite[Corollary III.11.5]{Hartshorne-algebraic}) 
and the fact that a rational surface is simply connected imply that any morphism $f$ from a rational surface $S$ onto 
$\PP^1$ factorizes through a morphism $g:S\to \PP^1$ with connected fibers and a covering $c$ of $\PP^1$, that is, 
$f=c \circ g$. In other words, any pencil whose resolution does not result into connected fibers can be refined
(after a base change) into a pencil with connected fibers (also known as \emph{primitive pencil}).

From a purely combinatorial point of view one can ask themselves if any combinatorial pencil admits a refinement
into a primitive combinatorial pencil. 

Similar results for line arrangements already exist (see~\cite{Marco-admissible,Falk-Yuzvinsky-multinets}).

\subsection{Weak combinatorics of resolutions}
In this section we will construct a combinatorial analogue of a bolwing-up process, which will lead to the 
concept of solvable ACC. Such combinatorics have the appropriate structure for our purpose.

\begin{definition}
Let $W=(\mathbf r, S, \Delta, \partial, \phi, \mu)$ be an ACC. Let us fix a point $P\in S$ and a list 
$\bar \nu:=(\nu_\delta)_{\delta\in \Delta_P}$ of positive integers. We say that the ACC
$\widehat W=(\hat{\mathbf r},\hat S,\hat \Delta,\hat \partial, \hat \phi,\hat \mu)$ is obtained as 
a $\sigma$-process at $P$ from $W$ (denoted by $W \leftarrow \hat W$) if there exists a partition 
$\{\hat P_1,...,\hat P_\ell\}$ of $\Delta_P$ such that the following properties are held:
\begin{enumerate}
\label{eq-mu-bup}
\item $\hat{\mathbf r}=\mathbf r \cup \{E\}$ (intuitively, $\hat{\mathbf r}$ results from adding the 
\emph{exceptional divisor} $E$ to $\mathbf r$),
\item $\hat S=(S\setminus \{P\}) \cup \{\hat P_1,...,\hat P_\ell\}$, (the point $P$ is replaced by the
\emph{infinitely near points} $\{\hat P_1,...,\hat P_\ell\}$),
\item $\hat \Delta=\Delta \cup \{\hat \delta_1,...,\hat \delta_\ell\}$, where $\hat \phi (\hat \delta_i)=E$,
$\hat \partial (\hat \delta_i)=\hat P_i$, $\hat \phi|_\Delta=\phi$, 
$\hat \partial|_{(\Delta\setminus \Delta_P)}=\partial|_{(\Delta\setminus \Delta_P)}$, and
$\hat \partial(\delta)=\hat P_i$ if $\delta\in \hat P_i$ (the exceptional divisor contributes with one local branch 
at each infinitely near point and the maps are naturally extended),
\item $\hat \mu(\delta,\hat \delta_i)=\nu_i$ if $\delta\in \hat P_i$ (this is the intersection multiplicity of
each branch with the exceptional divisor),
\item $\hat \mu(\delta_1,\delta_2)=\mu(\delta_1,\delta_2) - \nu_{\delta_1}\nu_{\delta_2}$, if 
$\delta_1,\delta_2\in \Delta_P$ and $\phi(\delta_1)\neq \phi(\delta_2)$, (the intersection multiplicity of two local 
branches after blow up decreases by the product of the multiplicities of the branches),
\item $\hat \mu$ extends $\mu$ outside $\SP^2(\Delta_P)$.
\end{enumerate}
\end{definition}

\begin{remark}
Let $W$ be the weak combinatorial type of a curve $\cC$ in a rational surface $V$ and let $V\leftarrow \hat V$ be 
a blow-up of $V$ at a singular point $P$ of $\cC$. Note that then the weak combinatorial type of the total 
transform $\hat \cC$ is obtained by a $\sigma$-process at $P$ from $W$ by using as $\bar \nu$ the list of 
multiplicities of the local branches at $P$. 
\end{remark}

This way one can extend the concept of resolution to general ACC's.

\begin{definition}
\label{def-res}
A sequence of $\sigma$-processes $W=W_0 \leftarrow W_1 \leftarrow ... \leftarrow W_n$ of ACC's
$W_k:=(\mathbf r^{(k)},S^{(k)},\partial^{(k)},\Delta^{(k)},\phi^{(k)},\mu^{(k)})$, 
$k=0,1,...,n$ is called a \emph{resolution} of $W$ if:
\begin{enumerate}
\item 
\label{def-res-prop-0}
$\nu_{\delta}\leq 1$ if $\delta\notin \Delta^{(0)}$, where $\nu_{\delta}$ is the multiplicity associated 
with $\delta$ at a $\sigma$-process, and
% \item
% \label{def-res-prop-1} For all $i,j\in \mathbf r^{(0)}$ one has
% $\sum_{{\tiny\begin{matrix} P \in S^{(n)}\end{matrix}}} \mu^{(n)}_P(i,j)=0,$ and
\item\label{def-res-prop-2} 
(Normal-crossing condition)
$\# \Delta^{(n)}_P = 2$ for any $P\in S^{(n)}$ and $\mu^{(n)}$ only takes values in $\{0,1\}$.
\end{enumerate}
An ACC is called \emph{solvable} if there exists a resolution.
\end{definition}

\begin{remark}
Note that the ACC obtained from a curve in $\PP^2$ admits a (combinatorial) resolution given by any (geometric)
resolution of its singularities, that is, every weak combinatorial type is solvable. Such a resolution will be
called a \emph{geometric} resolution of $W$. Note that weak combinatorial types might admit non-geometric
resolutions aswell.
\end{remark}

\subsection{Admissibility conditions}
Let $W$ be an ACC and let $\{v_i\}_{i\in \mathbf{r}}$ be a list of vectors in $\KK^{k}$.
For any $\delta\in \Delta$ define 
\begin{equation}
\label{eq-vdelta}
v_{\delta}:=\sum_{{\tiny{\begin{matrix}j\in \mathbf{r} \\ j\neq \phi(\delta)\end{matrix}}}}\mu(\delta,j)\ v_{j},
\end{equation}
where $\mu(\delta,j):=\sum_{\delta'\in \phi^{-1}(j)} \mu(\delta,\delta')$.
Note that, by Definition~\ref{def-acc}.(\ref{def-acc-cond3}), the only branches that contribute to $v_{\delta}$
are those in $\Delta_P$.

We say that $\{v_i\}_{i\in \mathbf{r}}$ satisfies the \emph{admissibility conditions} for $W$ if:
\begin{equation}
\label{eq-adm-cond}
\{v_{\phi(\delta)},v_{\delta}\}, \quad \text{ are linearly dependent for all } \delta\in \Delta.
\end{equation}
We will often denote this by saying $v_{\phi(\delta)}||v_{\delta}$ (note that one of the 
vectors might be zero).

\begin{definition}
A list of vectors $v_W:=\{v_i\}_{i\in \mathbf{r}}$ in $\KK^{k}$ satisfying the 
admissibility conditions~(\ref{eq-adm-cond}) for $W$ and spanning $\KK^{k}$ is called a 
\emph{$k$-admissible family} for $W$.
\end{definition}

One has the following result:

\begin{proposition}
\label{prop-comb-kadm}
If $(\cF,\bar m)$ is a combinatorial pencil of $(k+1)$-fibers of $W$, then there exists a
$k$-admissible family for $W$.
\end{proposition}

\begin{proof}
Let us consider $\cF=\{F_0,F_1,...,F_k\}$ and define the following family of vectors $v_i$, $i\in \mathbf r$:
$$v_i:=
\begin{cases}
m_i e_j & \text{ if } i\in F_j, j\neq 0 \\
-m_i (e_1+...+e_k) & \text{ if } i\in F_0. \\
\end{cases}
$$
Under these conditions note that if $P\in S$ is not a base point, then condition~(\ref{eq-adm-cond}) is
immediately satisfied since all the vectors involved are linearly dependent. Now, if $P\in S$ is a base point, 
then the condition~(\ref{def-cond2}) in Definition~\ref{def-comb-pencil} above implies that 
$k_{\delta} v_i+v_\delta=0$ and hence condition~(\ref{eq-adm-cond}) is also true.
\end{proof}

\begin{definition}
\label{def-k-adm-pencil}
The $k$-admissible family for $W$ associated with the combinatorial pencil $(\cF,\bar m)$ as in 
Proposition~\ref{prop-comb-kadm} 
will be referred to as \emph{the admissible family of $W$ associated with $(\cF,\bar m)$}.
\end{definition}

We need the following result from linear algebra.

\begin{lemma}
\label{lemma-adm-comd}
Suppose $v_1,\dots,v_r\in \KK^{k}$ are vectors such that
$$
\array{l}
v_1 || \sum_{j=1}^r a_{1,j} v_j, \\
v_2 || \sum_{j=1}^r a_{2,j} v_j, \\
\dots \\
v_{r-1} || \sum_{j=1}^r a_{r-1,j} v_j, \\
\endarray
$$
where $a_{i,j}=a_{j,i}$. Then 
$$
v_{r} || \sum_{j=1}^r a_{r,j} v_j.
$$
\end{lemma}

\begin{proof}
Note that if $v_r=0$ or $v_1=v_2=\dots=v_{r-1}=0$, then the result is immediate. 
Also, if any of the vectors $\{v_i\}_{i=1,\dots,r-1}$ are trivial, say $v_1=0$, it is
enough to solve the same problem on the remaining vectors, since the coefficients
$a_{1,i}=a_{i,1}$ either multiply the vector $v_1$ (and hence they have no contribution)
or multiply $v_i$ in the condition $v_1 || \sum_{j=1}^r a_{1,j} v_j$, which is 
trivially satisfied since $v_1=0$.

Therefore, we will assume that all the vectors $v_i$ ($i=1,\dots,r$) are non-zero.
Note that if $v\neq 0$, then $v || w$, implies the existence of $\lambda \in \KK$
such that $\lambda v+w=0$. Hence, in our case, there exist $\lambda_i\in \KK$ such that 
\begin{equation}
\label{eq-cmatrix}
\lambda_i v_i+\sum_{j=1}^r a_{i,j} v_j=0
\end{equation}
for $i=1,\dots,r-1$. Consider an $r\times r$ symmetric matrix $A:=(\alpha_{i,j})$ 
where the first $r-1$ rows are given by the coefficients of the 
equations~(\ref{eq-cmatrix}) over the variables $v_i$, and the last row is given 
by the coefficients of $\sum_{j=1}^{r-1} a_{r,j} v_j$ over the same variables. 
Also consider $V:=(v_{i,j})$ an $r \times k$ matrix whose $i$-th row is given 
by the coefficients of $v_i$. By~(\ref{eq-cmatrix}) one has that
$$A V=
\left(
\begin{matrix}
0 & 0 & \dots & 0 & 0 \\
0 & 0 & \dots & 0 & 0 \\
&&\dots &&\\
0 & 0 & \dots & 0 & 0 \\
b_1 & b_2 & \dots & b_{k-1} & b_{k}
\end{matrix}
\right),
$$ 
and hence 
$$V^t A V=\bar b^t v_r,$$
where $\bar b=(b_1,b_2,\dots,b_{k-1},b_{k})$ (row notation) is the vector of 
coordinates of $\sum_{j=1}^{r-1} a_{r,j} v_j$. Since $V^t A V=\bar b^t v_r$
is symmetric, one obtains that $v_r||\sum_{j=1}^{r-1} a_{r,j} v_j$ and 
therefore $v_r||\sum_{j=1}^{r} a_{r,j} v_j$.
\end{proof}

In other words, the admissibility conditions for each point are redundant.

In what follows we look into how the admissibility conditions change under a $\sigma$-process. 
Consider $v_W:=(v_i)_{i\in \mathbf{r}}$ a $k$-admissible family of vectors for $W$, 
and $\widehat W$ a $\sigma$-process of $W$ at $P$ associated with the multiplicity 
list $\bar \nu:=(\nu_\delta)_{\delta\in \Delta_P}$.

One has the following.

\begin{proposition}
\label{prop-adm-fam-blowup}
Let $W$ be an ACC and $\widehat W$ a $\sigma$-process of $W$. Then 
any $k$-admissible family $v_W$ on $W$ induces a $k$-admissible family 
$v_{\widehat W}$ on $\widehat W$ and vice versa.
\end{proposition}

\begin{proof}
Let $v_W$ be an admissible family on $W$. Then $v_{\widehat W}=(\hat v_i)_{i\in \mathbf{\hat r}}$ 
is defined as follows: $\hat v_i=v_i$ if $i\in \mathbf{r}$. The new vector associated
with the exceptional divisor $\hat v_E$ is defined as:
\begin{equation}
\label{eq-ve}
\hat v_E:=\sum_{\delta'\in \Delta_P} \nu_\delta v_{\phi(\delta')}.
\end{equation}
It remains to verify that the new admissible conditions are satisfied. In order to avoid
ambiguity, all the new vectors in $\hat W$ will be denoted as $\hat v$.
First we fix an infinitely near point $\hat P$, a branch $\delta\in \hat \Delta_{\hat P}$,
and assume $i:=\phi(\delta)$. We have two cases:
\begin{enumerate}
\item If $\delta\in \Delta_P$, then 
$$\hat v_\delta:=
\sum_{{\tiny{\begin{matrix}j\in \mathbf{\hat r} \\ j\neq i\end{matrix}}}}
\hat \mu(\delta,j)\ v_j=
\sum_{{\tiny{\begin{matrix}j\in \mathbf{r} \\ j\neq i\end{matrix}}}} 
\Big(\mu(\delta,j)-\nu_\delta 
\Big(\sum_{{\tiny{\begin{matrix}\delta'\in \Delta_P \\ \phi(\delta')=j\end{matrix}}}} 
\nu_{\delta'}\Big)\Big)\ v_{j} + \nu_{\delta} \hat v_{E},$$
(see~(\ref{eq-vdelta})). Therefore
\begin{equation}
\label{eq-vdeltahat}
\hat v_\delta=v_\delta + 
\nu_\delta \Big( \sum_{{\tiny{\begin{matrix}\delta'\in \Delta_P \\ \phi(\delta')=i\end{matrix}}}}
\nu_{\delta'}\Big)\ v_i.
\end{equation}
By hypothesis $v_\delta || v_i$, thus $\hat v_\delta|| v_i$.
\item If $\delta=\hat \delta$ is the branch of the exceptional divisor $E$ at $\hat P$.
Since there is only one such branch at $\hat P$, then Lemma~\ref{lemma-adm-comd} and the previous
case gives the result.
\end{enumerate}

The converse is immediate since according to~(\ref{eq-vdeltahat}) 
$v_\delta || v_i$ if and only if $\hat v_\delta|| v_i$.
\end{proof}

\subsection{Dicritical divisors and connectedness of an ACC}
\begin{definition}
Let $v_W$ be a $k$-admissible family of vectors. Consider $\widehat W$ 
a resolution of $W$ and $v_{\widehat W}$ the $k$-admissible family obtained from $v_W$ 
as in Proposition~\ref{prop-adm-fam-blowup}. An exceptional divisor $E$ is called a 
\emph{dicritical divisor} in $\widehat W$ if there exist at least two branches 
$\delta_1$ and $\delta_2$ in $\hat \Delta$ such that $D_1:=\phi(\delta_1)\neq \phi(\delta_2)=:D_2$ 
and $\hat \mu(\delta_1,E)=\hat \mu(\delta_2,E)=1$ and $\{v_{D_1},v_{D_2}\}$ 
generate a two-dimensional space. Also, an exceptional divisor $E$ is called \emph{trivial} if $v_E=0$.
\end{definition}

\begin{remark}
Under the conditions above, any dicritical divisor is trivial.
This is a direct consequence of the fact that $\widehat W$ must satisfy the admissibility 
conditions. Since $\widehat W$ is a resolution, one has that $v_{D_1}||v_E$ and $v_{D_2}||v_E$,
and $v_{D_1} \not|| v_{D_2}$ which is only true if $v_E=0$.

However, note that the converse is not true in general. For example, consider three 
smooth conics in a pencil of conics that are pairwise bitangent and use the vectors 
$v_1=(1,0)$, $v_2=(0,1)$, and $v_3=(-1,-1)$ for each conic. After the first blowing up of any 
base point one obtains an exceptional divisor $E$ which will not become a dicritical divisor,
but whose associated vector $v_E$ is zero, since $v_E=v_1+v_2+v_3$.
\end{remark}

\begin{definition}
We say two components of an ACC \emph{intersect}, if there are branches 
of each intersecting with positive multiplicity of intersection. Analogously, we say two
components $A$, $B$ are \emph{connected} if there is a sequence of components 
intersecting pairwise. Formally, $A$, $B$ are connected if there is a sequence of 
components $A_0=A,A_1,...,A_n=B$, a sequence of branches 
$\delta_i,\delta'_i\in \Delta$, ($i=1,...,n$) such that $\phi(\delta_i)=A_{i-1}$, $\phi(\delta'_i)=A_i$, 
satisfying $\mu(\delta_i,\delta'_i)\neq 0$. Therefore the concept of 
\emph{connected components} of an ACC can be defined.
\end{definition}

One has the following interesting result.

\begin{lemma}
\label{lemma-dic-connexion}
After resolution and after removing the trivial divisors, different fibers of a 
combinatorial pencil belong to different connected components.
\end{lemma}

\begin{proof}
Let $(\cF,\bar m)$ be a combinatorial pencil in $\cC$ and let $v_W$ be its associated 
$k$-admissible family. Consider $W\leftarrow \widehat W$ a resolution of singularities of
$\cC$. By Proposition~\ref{prop-adm-fam-blowup} the associated family of vectors 
$v_{\widehat W}$ also satisfies the admissibility conditions shown in~(\ref{eq-adm-cond}). 
Since $\widehat W$ corresponds to the combinatorics of a normal crossing divisor the 
condition at each normal crossing of two divisors, say $E$ and $E'$, means that either
$v_E=k v_{E'}$, $k\in \KK^*$, or $v_E=0$, or $v_{E'}=0$. In other words, the vectors
associated with $\widehat W$ are either multiples of the original vectors of $W$ or zero.
Moreover, after removing the trivial divisors, components from different fibers are disconnected.
\end{proof}

\subsection{Intersection matrix}
Let $W$ be a solvable ACC satisfying the B\'ezout conditions~(Definition~\ref{eq-wc}). 
Consider $W=W_0 \leftarrow W_1 \leftarrow \dots \leftarrow W_n$ a resolution of $W$, 
and $v_W:=(v_i)_{i\in \mathbf{r}}$ a $k$-admissible family of vectors. 
Denote by $P_\ell$ the point blown-up at each step $W_\ell \leftarrow W_{\ell+1}$ 
with multiplicity list 
$\bar \nu^{(\ell)}:=(\nu_\delta^{(\ell)})_{\delta\in \Delta_P}$, 
and define by $\cP$ the collection of such points. As seen in Proposition~\ref{prop-adm-fam-blowup},
associated with $v_W$ there are $k$-admissible families 
$v_{W_\ell}:=(v_i)_{i\in \mathbf{r}^{(\ell)}}$ of $W_\ell$. 
Define $\mathbf{\tilde r}^{(n)}:=\{i\in \mathbf{r}^{(n)} \mid v_i\neq 0\}$.
Consider the following \emph{incidence matrix} associated with the resolution $W_n$ of $W$ and with
the admissible family $v_W$:
$$J:=(a_{ij})_{i\in {\mathbf{\tilde r}}^{(n)}, j\in \cP},$$
where 
$$a_{ij}:=
\begin{cases}
\nu_{i}^{(\ell)} & \text{ if } j=P_\ell, i\in \mathbf{r}^{(\ell)} \\
-1 & \text{ if } j=P_\ell, i=E_{\ell+1}\\
0 & \text{ otherwise,}
\end{cases}
$$
(recall that 
$\nu_{i}:=\sum_{{\tiny{\begin{matrix}\delta\in \Delta_P\\ \phi(\delta)=i\end{matrix}}}} \nu_{\delta}$ 
and note that $\# \cP=n$).
Also define the \emph{degree matrix} $D:=\bar d^t \bar d$, where 
$\bar d:=(d_{i})_{i\in {\mathbf{\tilde r}}^{(n)}}$, and
$$d_{i}:=
\begin{cases}
d_i & \text{ if } i\in \mathbf r^{(0)}\\
0 & \text{ otherwise.}
\end{cases}
$$
Finally, we combine both matrices in order to define $Q:=D-JJ^t$.
Also, for convenience, if $W$ is already normal crossing~(\ref{def-res-prop-2}), then we set $JJ^t=0$.

Note that $Q$ is a square matrix of order $\# \mathbf{\tilde r}^{(n)}$.

\begin{proposition}
\label{prop-q}
If $W$ satisfies the B\'ezout Conditions and 
$W\leftarrow W_n=\widehat W=(\hat{\mathbf r},\hat S,\hat \Delta,\hat \partial,\hat \phi,\hat \mu)$ 
is a resolution of $W$, then the matrix $Q=(q_{ij})_{i,j\in \hat{\mathbf{\tilde r}}}$ satisfies the following:
\begin{enumerate}[$(1)$]
\item \label{prop-q-1}
$q_{ij}=\hat \mu(i,j):=
\sum_{{\tiny{\begin{matrix}\delta\in \hat \phi^{-1}(i)\\ \delta'\in \hat \phi^{-1}(j)\end{matrix}}}} 
\hat \mu(\delta,\delta')$ if $i\neq j$ 
(intuitively, the number of points at which the components $i$ and $j$ intersect in $\widehat W$). 
In particular $q_{ij}\geq 0$, 
\item \label{prop-q-2}
$q_{ii}=d_{i}^2-\sum_{\ell\in L_i} (\nu_i^{(\ell)})^2$, where 
$L_{i}:=\{\ell \mid i\in \phi^{(\ell)}(\Delta^{(\ell)}_{P_\ell})\}$.
\end{enumerate}

Moreover, if the resolution is geometric, then $Q$ is the intersection matrix of the non-trivial divisors of
$\hat{\mathbf{r}}$ in $\widehat W$.
\end{proposition}

\begin{proof}
By definition of $Q$ one has that 
$$
q_{ij}=d_id_j-
\sum_{\ell\in L_{ij}} \nu_i^{(\ell)}\nu_j^{(\ell)},
% =d_id_j-
% \sum_{\ell\in L_{ij}} 
% \sum_{{\tiny\begin{matrix}\phi^{(\ell)}(\delta_i)=i,\\ \phi^{(\ell)}(\delta_j)=j\end{matrix}}} 
% \nu_{\delta_i}^{(\ell)}\nu_{\delta_j}^{(\ell)},
$$
where $L_{ij}:=\{\ell \mid i,j\in \phi^{(\ell)}(\Delta^{(\ell)}_{P_\ell})\}$. This implies~(\ref{prop-q-2}).

By definition, one has $\mu^{(\ell)}(i,j)=\nu_i^{(\ell+1)}\nu_j^{(\ell+1)} + \mu^{(\ell+1)}(i,j)$. 
Also, if $W$ satisfies the B\'ezout Conditions then $d_id_j=\mu(i,j)$. Hence,
$d_id_j=\sum_{\ell\in L_{ij}} \nu_i^{(\ell)}\nu_j^{(\ell)} + \hat \mu(i,j)$ and~(\ref{prop-q-1}) follows. 
Finally, since $\widehat W$ is a resolution, $q_{i,j}$ is exactly the number of points at which the 
components $i$ and $j$ intersect in $\widehat W$.

The second part is a consequence of the Noether formula for the multiplicity of intersection
of branches after resolution.
\end{proof}

\begin{example}
\label{ex-conics}
Consider the combinatorics corresponding to the arrangement of four curves:
a smooth conic, two tangent lines to the conic, and the line joining the 
tangency points. We will order them and the exceptional divisors of the 
expected resolution as in Figure~\ref{fig-conic-line-pencil}.
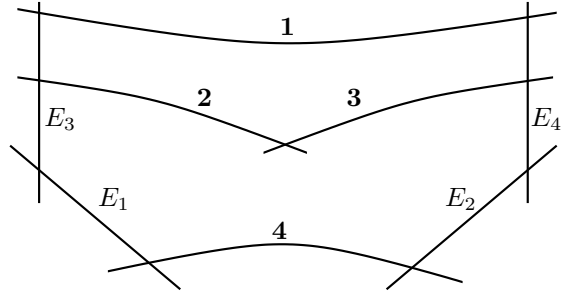
\begin{figure}[ht]
%\hspace*{3cm}
\includegraphics[scale=.3]{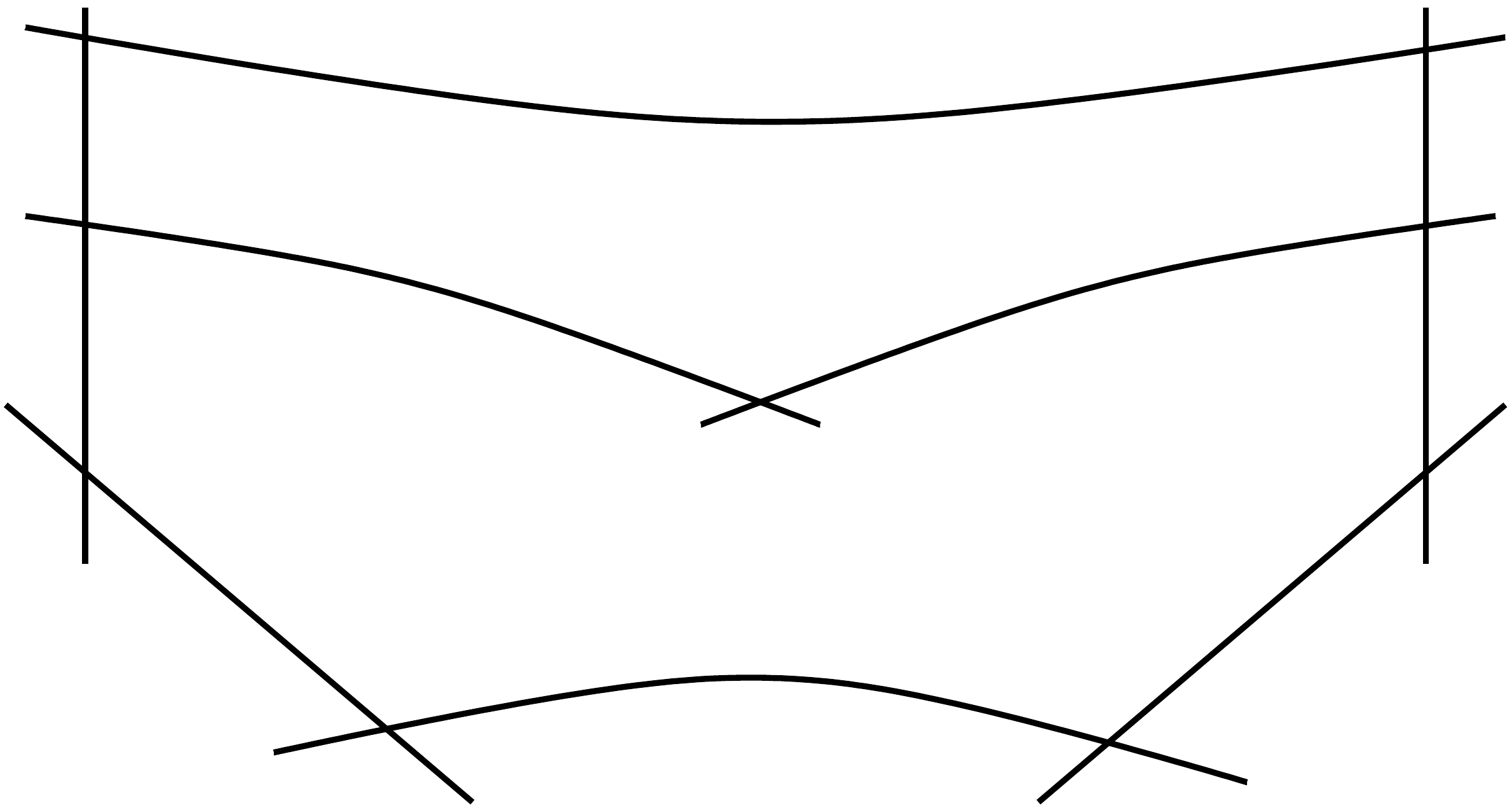}
\begin{picture}(0,0)
\put(-105,100){\makebox(0,0){$\mathbf 1$}}
\put(-170,35){\makebox(0,0){$E_1$}}
\put(-40,35){\makebox(0,0){$E_2$}}
\put(-190,65){\makebox(0,0){$E_3$}}
\put(-8,65){\makebox(0,0){$E_4$}}
\put(-108,23){\makebox(0,0){$\mathbf 4$}}
\put(-80,73){\makebox(0,0){$\mathbf 3$}}
\put(-136,73){\makebox(0,0){$\mathbf 2$}}
\end{picture}
\caption{Resolution of a conic-line arrangement}
\label{fig-conic-line-pencil}
\end{figure}
The only non-trivial case corresponds to when $E_3$ and $E_4$ are the only 
dicritical divisors. The corresponding matrices follow:
$$J:=
\left[
\begin{matrix}
1 & 1 & 1 & 1 \\ 
1 & 0 & 1 & 0 \\ 
0 & 1 & 0 & 1 \\ 
1 & 1 & 0 & 0 \\ 
-1 & 0 & 1 & 0 \\ 
0 & -1 & 0 & 1
\end{matrix}
\right],
D:=
\left[
\begin{matrix}

4 & 2 & 2 & 2 & 0 & 0 \\ 
2 & 1 & 1 & 1 & 0 & 0 \\ 
2 & 1 & 1 & 1 & 0 & 0 \\ 
2 & 1 & 1 & 1 & 0 & 0 \\ 
0 & 0 & 0 & 0 & 0 & 0 \\ 
0 & 0 & 0 & 0 & 0 & 0

\end{matrix}
\right], \text{ and } 
Q:=
\left[
\begin{matrix}

0 & 0 & 0 & 0 & 0 & 0 \\ 
0 & -1 & 1 & 0 & 0 & 0 \\ 
0 & 1 & -1 & 0 & 0 & 0 \\ 
0 & 0 & 0 & -1 & 1 & 1 \\ 
0 & 0 & 0 & 1 & -2 & 0 \\ 
0 & 0 & 0 & 1 & 0 & -2

\end{matrix}
\right].
$$
\end{example}

\begin{definition}
Let $\mathcal P_1:=(\mathcal F_1,\bar m_1)$ and $\mathcal P_1:=(\mathcal F_2,\bar m_2)$ be 
combinatorial pencils of $W$. We say $\mathcal P_2$ is a \emph{refinement} of $\mathcal P_1$ if 
the fibers of $\mathcal P_2$ are contained in the fibers of $\mathcal P_1$.
\end{definition}

\begin{definition}
Let $W$ be a solvable ACC and let $W\leftmap{\pi} \widehat W$ be a resolution, we say 
a combinatorial pencil $\mathcal P=(\cF,\bar m)$ is \emph{primitive w.r.t. $\pi$} if each two components in the 
same fiber of $\cF$ can be connected in $\widehat W$ by non-trivial divisors and the multiplicities of the 
components are coprime, that is, $\gcd(\bar m)=1$.

If $W$ is a weak curve combinatorics, then we simply call a combinatorial pencil \emph{primitive} if it is
pimitive w.r.t. a geometric resolution.
\end{definition}

Note that the combinatorial notion of primitive is again combinatorial, since the resolution is given
again combinatorially.

\begin{theorem}
\label{thm-prim-ref}
Let $W$ be a solvable ACC satisfying B\'ezout Conditions and let $W\leftmap{\pi} \widehat W$ be a resolution,
then any combinatorial pencil in $W$ admits a primitive refinement w.r.t $\pi$.
\end{theorem}

\begin{proof}
According to Proposition~\ref{prop-comb-kadm} any combinatorial pencil admits a
$k$-admissible family $v_W:=\{v_i\}_{i\in\mathbf{r}}$. Take a resolution 
$W=W_0\leftarrow W_1 \leftarrow ... \leftarrow W_n=\widehat W$ and 
consider the resulting admissible family $v_{\widehat W}$.

By Lemma~\ref{lemma-dic-connexion}, there exist trivial divisors. After removing them, 
one can construct the matrices $J$, $D$ and $Q$ as above. Let $A$ be the matrix 
whose columns are the non-zero vectors of $v_{\widehat W}$. It is easy to check that B\'ezout's Theorem
implies $D\cdot A^t=0$, and, by construction, $J^t \cdot A^t=0$. So the rows of $A$ are in 
the kernel of both $Q$ and $D$. Note that the kernel of $JJ^t$ coincides with the kernel of $J^t$.

After a suitable ordering of the elements of $\tilde{\mathbf{r}}^{(n)}$, one can assume that 
the matrix $Q$ decomposes into a direct sum of irreducible boxes $Q_\gl$, $\gl=1,\dots,\gk$. 
This induces a partition $\tilde \cF:=\{\tilde F_1,\dots,\tilde F_\gk\}$, (with 
$\tilde F_\gl\subset \tilde{\mathbf{r}}^{(n)}$) 
of the components of the combinatorics, and hence, 
in the columns of $A$. Each submatrix $Q_\gl$ is symmetric, and by 
Proposition~\ref{prop-q}(\ref{prop-q-1}) 
it has non-negative entries outside the diagonal. Hence, using the Vinberg
classification of matrices (see~\cite[Thm. 4.3]{Kac-infinite}) on $-Q_\lambda$, one can ensure that
$Q_\gl$ is of one of the following types:
\begin{itemize}
\item[(Fin)] 
$\det(Q_\gl)\neq 0$; there exists a vector $v$ with positive entries 
such that $Q_\gl v$ has negative entries.
\item[(Aff)] 
$\corank(Q_\gl)=1$, and its kernel is generated by a vector with only positive entries.
\item[(Ind)] 
There exists a vector $u$ with only positive entries such that $Q_\gl u$ has only positive entries.
\end{itemize}
Here (Fin), (Aff), and (Ind) stand for Finite, Affine, and Indefinite types respectively.
We have seen that $Q$ has a nontrivial kernel, so the $Q_\gl$ cannot be an (Fin)-matrix. 
If one of the $Q_\gl$, say $Q_1$, is an (Ind)-matrix, one can consider a vector $u_1$ with 
positive entries such that $Q_1u_1$ has only positive entries. For the rest of the $Q_\gl$, 
one can find vectors $u_\gl$ with only negative entries such that $Q_\gl u_\gl$ has only zero 
entries (if $Q_\gl$ is an (Aff)-matrix) or only negative entries (if $Q_\gl$ is an (Ind)-matrix). 
By multiplying the $u_\gl$ by suitable positive numbers, one can reconstruct a vector $u$ 
such that $Du=0$.

Now denoting by $(\cdot,\cdot)$ the standard scalar product:

$$
0\geq-(J^tu,J^tu)=(Qu,u)-(Du,u)=(Qu,u)=(Q_1u_1,u_1)+\sum_{i\geq 2}(Q_\gl u_\gl,u_\gl)
\geq (Q_1u_1,u_1)>0,
$$
which leads to contradiction. So we can conclude that all the $Q_\gl$ are (Aff)-matrices.

Note that, by Proposition~\ref{prop-q}, the partition $\tilde \cF$ induced by the boxes $Q_\gl$ 
is equal to the partition given by the connected components in $\mathbf{\tilde r}^{(n)}$.

Note that a vector is in the kernel of $Q$ if and only if it is made up of vectors that are 
in the kernel of the $Q_\gl$'s. In particular, the kernel of $Q$ has dimension equal to the 
number of irreducible boxes.

From now on, $K_Q$ will denote the kernel of $Q$, and $K_D$ will denote the kernel of $D$, the degree matrix. 
Let $u_\gl$ be a positive vector that generates the kernel of the box $Q_\gl$, and $\tilde u_\gl$ 
the vector of $\KK^{r_n}$, $r_n:=\# \mathbf{\tilde r}^{(n)}$ obtained from $u_\gl$ by completing the 
entries corresponding to the other boxes with zeroes. As mentioned above, 
$\{\tilde u_\gl\}_{\gl=1,\dots,\gk}$ is a basis of $K_Q$. Since $D=\bar d^t \cdot \bar d$, thus 
$K_D=\ker \bar d$, and thus $\codim K_D=1$. Also, since $K_Q$ has a set of non-negative vectors as a 
basis, one has that $K_Q\not \subset K_D$, and thus $\dim K_Q\cap K_D=\dim K_Q-1$. 
By B\'ezout's Theorem $d:=\bar d \cdot \tilde u_{\gl}$ is independent of $\gl$ and hence 
$\{\tilde u_\gl-\tilde u_\gk\}_{\gl=1,\dots,\gk-1}$ is a basis of $K_Q\cap K_D$. 

Consider $N$, the matrix whose rows are the family of vectors
$\{\tilde u_\gl-\tilde u_\gk\}_{\gl=1,\dots,\gk-1}$. 
As we have seen before, the rows of $A$ must be a linear combination of the rows of $N$.

Let us construct the family of vectors $w:=\{w_i\}_{i\in\mathbf{r}^{(n)}}$ as follows:
\begin{itemize}
\item $w_i=0$ if $i\in\mathbf{r}^{(n)}$ is a trivial divisor.
\item Otherwise, $w_i$ is equal to the corresponding column of $N$.
\end{itemize}

\begin{lemma}
The set of vectors $w$ is a $(\kappa-2)$-admissible family for $\widehat W$.
\end{lemma}
\begin{proof}
Since $\widehat W$ is a normal crossing ACC, the admissibility conditions are verified if 
whenever two components, say $i$ and $j$ intersect, then $w_i||w_j$.

If one of them, say $i$, is a trivial divisor, then there is nothing to prove, since $w_i=0$.
If neither of them are trivial divisors, then $i,j\in F_\gl$ for some $\gl$ in the partition 
$\tilde \cF$. Therefore $v_i$ (and also $v_j$) 
must be proportional to $e_\lambda$ (the obvious vector of the canonical basis) when 
$\gl\neq \gk$ and to $(-1,\ldots,-1)$ if $\gl= \gk$. Thus the result follows.
\end{proof}

All this shows that $\rank N=\dim (K_Q\cap K_D)$, and the columns of $N$ are multiples of 
the following $\gk$ vectors: the $\gk-1$ vectors $e_\gl$ of the canonical basis of 
$\ZZ^{\gk-1}$ and the vector $e_\gk:=(-1,\ldots,-1)$. There is essentially only one 
linear relation among them, which is $\sum_{\gl=1}^\gk e_\gl=0$. Moreover, if $i\in F_\gl$, 
then $v_i=m_i \cdot e_\gl$, where $m_i$ is a positive rational number. Note that, without 
loss of generality, after multiplication by a natural number (for instance the least common
multiple of the denominators) one can assume that the $m_i$'s are integer numbers.

Let us show that $\cF:=\tilde \cF \cap \mathbf{r}$ and 
$\bar m:=(m_i)_{i\in \mathbf{r}}$
defines a combinatorial pencil of $\gk$ fibers. In order to do so we need to check both 
properties in Definition~\ref{def-comb-pencil}. Let $P\in S$ not satisfying 
Property~\ref{def-comb-pencil}(\ref{def-cond1}) and let $\delta\in \Delta_P$ a branch of $i$ 
at $P$. Since $v$ is an admissible family, 
$\sum_{j\neq i} \mu(\delta,j)v_j= 
\sum_{\gl=1}^\gk \left( \sum_{j\in F_\gl} \mu(\delta,j) m_j \right) e_\gl$ 
is proportional to $v_i$, which forces the sum
$k_{\delta}:=\sum_{j\in F_\gl} \mu(\delta,j) m_j$ to be independent of $\gl$.
\end{proof}

\begin{remark}
By Proposition~\ref{prop-q}, in the case where the combinatorial pencil is realizable and the resolution
is geometric one can prove Theorem~\ref{thm-prim-ref} without using the Vinberg classification of matrices.
According to Proposition~\ref{prop-q} in this case $Q_\lambda$ is the intersection matrix of the divisors
in $F_\lambda$. Using Zariski's Lemma for fibrations of surfaces 
(see~\cite[Lemma III.8.2]{Barth-Peters-VandeVen-compct-complex})
a divisor $D=\sum m_iC_i$, where the divisors $C_i$ belong to the same fiber $F_\lambda$, satisfies 
$D^2=\bar m Q_\lambda \bar m^t=0$ if and only if $pD=qF_\lambda$ for $p,q\in \ZZ\setminus \{0\}$.
In particular, for line arrangements this provides another 
proof of the analogous result given in~\cite{Marco-admissible,Falk-Yuzvinsky-multinets}.
\end{remark}

\begin{example}
Let us further analyze Example~\ref{ex-conics}. The matrix $Q$ has three boxes $Q_1=(0)$, 
$Q_2=\left[ \begin{matrix} -1 & 1 \\ 1 & -1\end{matrix}\right]$, and 
$Q_3=\left[\begin{matrix} -1 & 1 & 1\\ 1 & -2 & 1 \\ 1 & 1 & -2\end{matrix}\right]$. 
Therefore $v_1=(1)$, $v_2=(1,1)$, and $v_3=(2,1,1)$ generate the kernels of $Q_1$, $Q_2$ and 
$Q_3$ and $v=(1,1,1,2,1,1)\in K_Q \cap K_D$ according to the notation above. 
In fact, if $\cF:=\{\{\mathbf 1\},\{\mathbf 2, \mathbf 3\},\{\mathbf 4\}\}$
and $\bar m:=(1,1,1,2)$, then $(\cF,\bar m)$ is a combinatorial pencil,
which corresponds to the geometric pencil generated by the conic and the
two lines, which contains the third line twice.
\end{example}

\section{Combinatorial Max Noether Fundamental Theorem}

We are ready to prove the following version of the Max Noether Fundamental Theorem.

\begin{theorem}
\label{thm-noether}
Let $V(F)=\{F=f_1^{\ell_1}\cdot ...\cdot f_p^{\ell_p}=0\},V(G)=\{G=g_1^{m_1}\cdot ...\cdot g_q^{m_q}=0\}$ 
and $V(H)=\{H=h_1^{n_1}\cdot ...\cdot h_r^{n_r}=0\}$ be three curves.
If $F,G$ and $H$ are in a primitive combinatorial pencil, then $F,G$ and $H$ are in a pencil.
\end{theorem}

\begin{proof} 
Denote by $\cC\subset \PP^2$ the union of $V(F)$, $V(G)$ and $V(H)$ and by $S$ the set of base points of 
the combinatorial pencil. Let us assume that all irreducible components of $H$ can be connected outside
the combinatorial pencil. Consider a local branch $\gamma$ of $H$ at a point $P\in S$. Denote
$k:=\mu_P(\gamma,F)=\mu_P(\gamma,G)$. Note that generically, one has that 
$\mu_P(\gamma,\alpha F+\beta G)=k$, but there exists a point $[\alpha_\gamma:\beta_\gamma]\in \PP^1$ 
such that $\mu_P(\gamma,\alpha_\gamma F+\beta_\gamma G)>k$. Let us denote by $h_1$ the irreducible 
component containing $\gamma$. Note that, 
$$\begin{cases}
\mu_Q(h_1,\alpha_\gamma F+\beta_\gamma G)
>\mu_Q(h_1,F)=\mu_Q(h_1,G) & \text{ if } Q=P\\
\mu_Q(h_1,\alpha_\gamma F+\beta_\gamma G)
\geq \mu_Q(h_1,F)=\mu_Q(h_1,G) & \text{ otherwise.}
\end{cases}
$$
Therefore, by B\'ezout, 
$d \cdot \deg h_1 = \sum_{Q\in S} \mu_Q(h_1,\alpha_\gamma F+\beta_\gamma G) > \sum_{Q\in S} \mu_Q(h_1,F)=d \cdot \deg h_1$, 
which implies that $\alpha_\gamma F+\beta_\gamma G$ is a multiple of $h_1$.
Consider now a resolution of the pencil, obtained by blowing up the base points $\pi: X\to \PP^2$. The morphism 
$f:X\to \PP^1$ is now well defined on the rational surface $X$ and if $P\notin S$ then $f(\pi^{-1}(P)):=[F(P):G(P)]$.
By hypothesis, the preimage of $V(H)$ outside the dicritical divisors is connected and hence $f$ is constant on the 
strict transform of $V(H)$, for any $Q\in \pi^{-1}(V(H))$ one has that $f(Q)=[-\beta:\alpha]\in \PP^1$.
On the other hand we know that for any $h_1$ and $h_2$ components of $H$ one has that $\alpha_1 F+\beta_1 G = h_1 u_1$ and 
$\alpha_2 F+\beta_2 G = h_2 u_2$. 

Consider now $P_1$ (resp. $P_2$) a regular point of $V(h_1)\setminus S$ (resp. $V(h_2)\setminus S$) and $Q_i:=\pi^{-1}(P_i)$. 
By the remarks in the previous paragraph, $f(Q_i):=[F(P_i):G(P_i)]=[-\beta:\alpha]$. Therefore 
$[-\beta_1:\alpha_1]=[-\beta_2:\alpha_2]=[-\beta:\alpha]$, and hence $\alpha F + \beta G=H'K$, where
$H'=h_1^{n'_1}\cdot h_2^{n'_2}\cdot...\cdot h_r^{n'_r}$. We will denote this by $H'K_H\in (F,G)$.

Let us denote by $Q_1$ (resp. $Q_2$) the intersection matrix associated with a resolution of $F\cdot G\cdot H$ 
(resp. $F\cdot G\cdot H'\cdot K$) that dominates both. According to Proposition~\ref{prop-q}, $Q_1$ and $Q_2$ have 
the following form 

$$
Q_1=
\left[
\begin{matrix}
Q_F & 0 & 0 \\
0 & Q_G & 0 \\
0 & 0 & Q_H
\end{matrix}
\right]
\quad \quad
Q_2=
\left[
\begin{matrix}
Q_F & 0 & 0 & 0\\
0 & Q_G & 0 & 0\\
0 & 0 & Q_{H'} & M \\
0 & 0 & M & Q_K
\end{matrix}
\right]
$$
where $Q_H=Q_{H'}$. Note that the box $Q_H$ has a kernel of dimension 1 and hence 
$M=0$, or else, the box 
$\tilde Q:=
\left[
\begin{matrix}
Q_H & M \\
M & Q_K
\end{matrix}
\right]
$ 
would have a vector of type $(v_H,0)\in \ker \tilde Q$, which contradicts $\tilde Q$ 
being of Affine type. Therefore we can assume that $M=0$, $H'=(H'')^q$, $H=(H'')^p$, and
$$
Q_2=
\left[
\begin{matrix}
Q_F & 0 & 0 & 0\\
0 & Q_G & 0 & 0\\
0 & 0 & Q_H & 0 \\
0 & 0 & 0 & Q_K
\end{matrix}
\right]
$$
which implies that the preimage of $H''$ and $K$ by the resolution $\pi$ are disconnected outside the
dicritical divisors. By the algebraic Stein Factorization Theorem one can find a refinement of the pencil 
$(F,G,H'K)$ into a pencil $(\tilde F,\tilde G,\tilde H)$. Since the original pencil is primitive, one has 
that $\tilde F^{n_F}=F$, $\tilde G^{n_G}=G$, and $\tilde H^{n_{H'}}=H'=(H'')^q$, that is, there exists 
$n_{H''}=\frac{n_{H'}}{q}\in \NN$ such that $\tilde H^{n_{H''}}=H''$. 
By the hypothesis on degrees, this implies that $n_F=n_G=p n_{H''}$. 
Since $\gcd (n_F,n_G,p n_{H''})=1$ by hypothesis, then $n_F=n_G=p n_{H''}=1$. Thus, $H=H'$ and therefore 
$\alpha F+\beta G=H$.

\end{proof}

As an immediate corollary one has the following.

\begin{corollary}
Let $F,G,H$ be three homogeneous polynomials of the same degree in three variables such that their zero sets
$V(F), V(G)$, and $V(H)$ are three irreducible curves with no common components. 
If $V(F)\cap V(G)=V(F)\cap V(H)=V(G)\cap V(H)=\{P\}$ and $F$, $G$, and $H$ are locally irreducible at $P$, then
$H=\alpha F+\beta G$.
\end{corollary}

\bibliographystyle{amsplain}
%\bibliography{biblio-ji}
%\input{max-noether.bbl}

\end{document}